\documentclass{amsart}
\usepackage{amsmath,amssymb,amscd,amsthm,verbatim,alltt,amsfonts,array}
\usepackage{mathrsfs}
\usepackage[english]{babel}
\usepackage{latexsym}
\usepackage{amssymb}
\usepackage{euscript}
\usepackage{graphicx}
\newtheorem{thm}{Theorem}[section]

\newtheorem{prop}[thm]{Proposition}

\newtheorem{lem}[thm]{Lemma}
\newtheorem{Def}[thm]{Definition}

\newtheorem{ex}{Example}[section]
\newcommand{\be}{\begin{equation}}
\newcommand{\ee}{\end{equation}}
\newcommand{\ben}{\begin{enumerate}}
\newcommand{\een}{\end{enumerate}}
\newcommand{\beq}{\begin{eqnarray}}
\newcommand{\eeq}{\end{eqnarray}}
\newcommand{\beqn}{\begin{eqnarray*}}
\newcommand{\eeqn}{\end{eqnarray*}}

\title[Generalized weakly-Weyl Finsler metrics]{Generalized Weakly-Weyl Finsler Metrics\\ A Generalized Approach to Sakaguchi's Theorem}
\author{Nasrin Sadeghzadeh$^*$}
\newcommand{\acr}{\newline\indent}
\address{\llap{*\,}Department of Mathematics,\acr
University of Qom, \acr
Alghadir Bld\acr
Qom\acr
Iran}
\email{nsadeghzadeh@qom.ac.ir}
\author{Meshkat Yavari}
\address{\llap{b\,}Department of Mathematics,\acr
University of Qom, \acr
Alghadir Bld\acr
Qom\acr
Iran}
\email{meshkat.yavari@stu.qom.ac.ir}

\begin{document}
\maketitle
\begin{center}
\textit{This is the accepted version of the paper published in \textbf{Differential Geometry and its Applications}, Vol.~101 (2025), 102297.}
\end{center}
\section{Abstract}
{The development of projective invariant Weyl metrics in this paper offers a fresh perspective, as we establish the characteristics of both weakly-Weyl and generalized weakly-Weyl Finsler metrics. We thoroughly examine the connections between these metrics and various projective invariants, highlighting their significance in the context of generalized Sakaguchi's Theorem, which states that every Finsler metric of scalar flag curvature is a GDW-metric. Additionally, we introduce several illustrative examples pertaining to this new class of projective invariant Finsler metrics.
Specifically, we explore the category of weakly-Weyl spherically symmetric Finsler metrics in $\mathbb{R}^n$. Importantly, we demonstrate that the two
classes weakly-Weyl and $W$-quadratic spherically symmetric Finsler metrics in $\mathbb{R}^n$ are equivalent.}
\\
\\
\subjclass \textbf{MSC [2020]}: {53B40; 53C60}\\
\keywords \textbf{Keywords}:{Weyl metrics, Weakly-Weyl metrics, Generalized weakly-Weyl metrics, Sakaguch's Theorem.}
\maketitle
\section{Introduction}
The birth of a new projective invariant in any space marks a significant advancement in the field of geometry.  Hermann Weyls pioneering work in the 1920s laid the foundation for the concept of Weyl curvature\cite{HermanWeyl}, a projective invariant tailored for Riemannian spaces, significantly influencing the advancement of Finsler geometry, a discipline with roots in the early 20th century \cite{PaulFinsler}.

In the 1920s, J. Douglas extended Weyl curvature to Finsler metrics, introducing the concept of Weyl and Douglas metrics. Weyl Finsler spaces are a special class of Finsler spaces that are characterized by the vanishing of the projective Weyl curvature.

The introduction of Weyl Finsler spaces has opened up new avenues for research and exploration in the field, making them a valuable tool for understanding geometric structures and their applications in various scientific disciplines.

The investigation of Weyl curvature in Finsler geometry persisted, leading Z. Szab\'{o} to confirm in the 1970s that Weyl metrics are precisely Finsler metrics with scalar flag curvature \cite{Weyl=Scalar}.
According to \cite{Sh2}, the Weyl curvature and Douglas curvature are considered among the fundamental projective invariants in Finsler geometry. This means that If a Finsler metric meets the criteria for Weyl curvature or Douglas curvature disappearing, then any Finsler metric that is projectively equivalent to it will also meet these criteria. Other notable projective invariants include the generalized Douglas-Weyl metrics and the class of Finsler metrics satisfying a certain equation involving the Weyl and Douglas metrics \cite{GDW}, \cite{SadeDouglas}. \\

The generalization of important geometric quantities and their applications is a remarkable phenomenon. In the field of projective invariant quantities, a notable example of this generalization is the class of $GDW$-metrics, which encompasses the important class of projective invariant Finsler metrics, including Weyl metrics. This raises a natural and important question: are there any other generalizations of Weyl structures that have a significant relationship with other projective invariant Finsler metrics? In other words, is there a broader framework that encompasses Weyl metrics or other projective invariant Finsler metrics, and if so, what are its implications for our understanding of geometric structures and their applications?

In this paper, we introduce a new perspective by developing projective invariant Weyl metrics. We define the attributes of both weakly-Weyl and generalized weakly-Weyl Finsler metrics, presenting a comprehensive approach to analyzing their relations with various types of other projective invariants. In particular, we establish the following theorems.
\begin{thm}\label{W-WeylProjInv}
Weakly-Weyl curvature is a projective invariant quantity in Finsler spaces.
\end{thm}
Additionally, we address the class of generalized weakly-Weyl Finsler metrics. Theorem \ref{W-WeylProjInv} can be understood as a direct consequence of Theorem \ref{GW-WeylProjInv} below. Specifically, since the class of weakly-Weyl Finsler metrics is a subclass of the generalized weakly-Weyl Finsler metrics, the projective invariance of the weakly-Weyl curvature follows naturally from the closure property of the broader class under projective changes. This logical dependency highlights that the invariance result presented in Theorem \ref{W-WeylProjInv} is subsumed by the more general statement on the stability of the entire generalized family in Theorem \ref{GW-WeylProjInv}.

\begin{thm}\label{GW-WeylProjInv}
The class of generalized weakly-Weyl Finsler metrics is closed under projective changes.
\end{thm}
The method of establishing generalized weakly-Weyl metrics is somewhat similar to the technique utilized in the study on \cite{SadeDouglas}, which introduced relatively isotropic $\tilde{D}$ metrics.
Our framework provides a unified approach to the study of projective invariant Finsler metrics, including Weyl metrics and other notable examples. This research leads to a different and more straightforward proof of Sakaguchi's Theorem, along with its generalization. We first prove the following Theorem.
\begin{thm} \label{WeylDouglasTheta}\label{PreSakaguchi}
For every Finsler metric $F$, with Weyl curvature $W=\{W_y\}_{y \in T_xM\setminus{0}}$ and Douglas curvature $D=\{D_y\}_{y \in T_xM\setminus{0}}$, it follows that
\be\label{WeylDotheta}
W_j{^i}_{ml.k}y^m=D_j{^i}_{kl|0} -\frac{1}{n+1} \theta_{jkl} y^i,
\ee
where $\theta_{jkl}=2E_{jk|l}-\frac{1}{3}(R{^s}_{l.s}-(n+2)R_{.l})_{.j.k}$.
\end{thm}
The subsequent theorem is a straightforward consequence of the previous one.
\begin{thm}\label{GSakaguchi}(\textbf{Generalized Sakaguchi Theorem})\\
Every weakly-Weyl Finsler metric is a $GDW$-metric.
\end{thm}
Additionally, we explore several examples within this new class of projective invariant Finsler metrics, focusing particularly on the class of weakly-Weyl spherically symmetric Finsler metrics.
\begin{thm}\label{W-WeylSphThm}
Every spherically symmetric Finsler metrics in $\mathbb{R}^n$ is weakly-Weyl if and only if it is of $W$-quadratic.
\end{thm}
The key question we explore in this paper is whether there exist other generalizations of Weyl structures that have a significant relationship with other projective invariant Finsler metrics. Our novel framework allows us to address this question by introducing a unified perspective on these important geometric structures.
\\
The symbols ${}_{.}$ and ${}_{|}$ in this article denote the vertical and horizontal derivatives with respect to the Berwald connection within the Finsler metric $F$. \\
Moreover, the subscript ${}_0$ represents the operation of contraction by $y^m$ as shown by the subscript ${}{m}$, while the symbol ${}_{;m}$ signifies differentiation in relation to $x^m$.\\
Additionally, the subscript ${}_{s}$ and ${}_{r}$ signify the derivative with respect to $s$ and $r$, respectively.
\section{Preliminaries}
On manifold $M$, a Finsler metric $F$ defined as a non-negative function on $TM$ must adhere to certain properties.
\ben
\item[(a)] $F$ is
 $C^{\infty}$ on $TM\setminus \{0\}$,
\item[(b)]
$F(\lambda y) =\lambda F(y)$, $\forall \lambda >0$, $\ y\in TM$,
\item[(c)] For each $y\in T_xM$,
the following quadratic form ${\bf g}_y$ on $T_xM$ is positive definite,
\be
{\bf g}_y(u, v):= {1\over 2} \Big [ F^2(y+ s u + tv ) \Big ]\Big |_{s, t=0}, \ \ \ \ \ \ u, v\in T_xM.
\ee
\een
The Finsler metric \( F \) on \( M \) satisfies the condition that at each point \( x \in M \), \( F_x = F|_{T_xM} \) is an Euclidean norm if and only if \( \mathbf{g}_y \) is independent of \( y \in T_xM \setminus \{0\} \).
A geodesic is defined by the following equation
\be
{d^2 c^i\over dt^2} + 2 G^i (c(t), \dot{c}(t))=0,
\ee
with local functions $G^i(x, y)$ given on $TM$ as
\be
G^i(x, y):= \frac{1}{4} g^{il}(x, y) \{ \frac{\partial^2 F^2}{\partial x^k \partial y^l} y^k - \frac{\partial F^2}{\partial x^l}\},\quad y\in T_xM, \label{Gi}
\ee
and $G^i=G^i(x, y)$ are called the spray coefficients of $F$. Denoted by
\[
G=y^i\frac{\partial }{\partial x^i}-2G^i \frac{\partial}{\partial y^i},
\]
the associated spray to \( (M, F) \) generates geodesics in \( M \), which correspond to the projections of the integral curves of \( G \). The Riemann curvature tensor \( R_y = R^{i}_{k} \frac{\partial}{\partial x^{i}} \otimes dx^{k} \) characterizes the curvature of \( F \) is given by
\[
R{^i}_k=2\frac{\partial G^i}{\partial x^k}-\frac{\partial^2 G^i}{\partial x^m \partial y^k}y^m+ 2G^m\frac{\partial^2 G^i}{\partial y^m \partial y^k}-\frac{\partial G^i}{\partial y^m}\frac{\partial G^m}{\partial y^k}.
\]
With respect to the Riemann curvature of the Finsler metric \( F \), it can be noted that \cite{Sh2}
\be\label{Rikl}
R{^i}_{kl}=\frac{1}{3}(R{^i}_{k.l}-R{^i}_{l.k}), \quad and \quad R_j{^i}_{kl}=R{^i}_{kl.j}.
\ee
Here, $"{}_{.k}"$ denotes the differential with respect to $y^k$.\\
When \( G^i(x, y) \) exhibits quadratic properties with respect to \( y \in T_xM \) for all \( x \in M \), the metric \( F \) is referred to as a Berwald metric. New, define
\[
B_y:T_xM\otimes T_xM\otimes T_xM\rightarrow T_xM
\]
\[
B_y(u,v,w)=B_j{^i}_{kl}u^j v^k w^l \frac{\partial}{\partial x^i},
\]
where
$
B_j{^i}_{kl}=\frac{\partial^3 G^i}{\partial y^j \partial y^k \partial y^l}$ and $u=u^i\frac{\partial}{\partial x^i}$, $v=v^i \frac{\partial}{\partial x^i}$, $w=w^i\frac{\partial}{\partial x^i}$.
The connection between Riemann and Berwald curvature is a topic of significant importance, as emphasized in \cite{Sh2}.
\be\label{RieBer}
B_j{^i}_{ml|k}-B_j{^i}_{mk|l}=R_j{^i}_{kl.m}.
\ee
Define
\[
E_y:T_xM\otimes T_xM \rightarrow \mathbb{R},
\]
\[
E_y(u,v)=E_{jk} u^j v^k,
\]
where $E_{jk}=\frac{1}{2}B_j{^m}_{km}$. The Berwald curvature, denoted by \( B \), and the mean Berwald curvature, represented by \( E \), are fundamental concepts in the study of Finsler metrics. When both \( B = 0 \) and \( E = 0 \), the Finsler metric \( F \) is classified as a Berwald metric, and it is also identified as a Weakly Berwald (WB) metric \cite{Sh3}. \\
Considering a Finsler metric \( (M, F) \) that exhibits isotropic mean Berwald curvature, if
\[
E_{ij}=\frac{n+1}{2}cF^{-1} h_{ij},
\]
is valid for some scalar function $c=c(x)$ on $M$, where \( h_{ij} \) represents the angular metric. The \( S \)-curvature \( S(x, y) \) is defined as follows \cite{Sh3}
\[
S(x,y)=\frac{d}{dt}[\tau(\gamma(t),\gamma'(t)]_{|t=0},
\]
where \( \tau(x, y) \) denotes the distortion of the metric \( F \), and \( \gamma(t) \) is the geodesic such that $\gamma(0)=x$ and $\gamma'(0)=y$ on $M$. It is widely accepted that \cite{Sh2}
\be\label{ES}
E_{ij}=\frac{1}{2}S_{.i.j}.
\ee
%where $.i$ denotes the vertical derivative with respect to Berwald connection of $F$, $\frac{\partial}{\partial y^i}$.\\
H. Akbar-Zadeh introduced the Finslerian value \( H \) to characterize Finsler metrics with constant flag curvature, which is derived from the mean Berwald curvature \( E \) through covariant horizontal differentiation along geodesics. For a vector \( y \in T_pM \), the function
\[
H_y : T_pM \times T_pM \longrightarrow R
\]
is defined as follows
\[
H_y(u, v) := H_{jk}(y)u^j v^k,
\]
where $H_{jk} := E_{jk|l}y^l$. %A Finsler metric $F$ is called H-metric if and only if $H = 0$.
Define
\be\label{Douglas}
D_j{^i}_{kl}=B_j{^i}_{kl}-\frac{1}{n+1}\frac{\partial^3}{\partial y^j \partial y^k \partial y^l}(\frac{\partial G^m}{\partial y^m}y^i).
\ee
The tensor $D:=D_j{^i}_{kl} dx^j\otimes \frac{\partial}{\partial x^i}\otimes dx^k \otimes dx^l$ is a well-defined tensor on the slit tangent bundle \( TM_0 \) and is referred to as the Douglas tensor. The Douglas tensor, denoted by \( D \), is a projective invariant that is non-Riemannian. This indicates that if two Finsler metrics \( F \) and \( \bar{F} \) are equivalent under projective transformations, meaning that
\be\label{ProjGeo}
G^i=\bar{G^i}+P y^i,
\ee
where the projective factor \( P = P(x, y) \) is positively \( y \)-homogeneous of degree one, then the Douglas tensor of \( F \) is identical to that of \( \bar{F} \) \cite{Sh2}.
It can be easily demonstrated that
\be\label{D2}
D_j{^i}_{kl}=B_j{^i}_{kl}-\frac{2}{n+1}\{E_{jk}\delta^i_l+E_{jl}\delta^i_k+E_{kl}\delta^i_j+E_{jk.l}y^i\}.
\ee
The Douglas curvature, denoted by $D_j{^i}_{kl}$, s a projective invariant constructed from the Berwald curvature. Finsler metrics for which \( D_j{^i}_{kl}= 0 \) are referred to as Douglas metrics. Moreover, metrics that satisfy the following equation are known as \( GDW \)-metrics, and they also remain invariant under projective transformations.
\[
{D_j}^i{}_{kl|m}y^m=T_{jkl}y^i,
\]
for some tensors \( T_{jkl} \), where \( D_j{^i}_{kl|m} \) denotes the horizontal derivatives of \( D_j{^i}_{kl} \) with respect to the Berwald connection of \( F \).
\begin{lem} \cite{Sh2}
Consider two Finsler metrics, $F$ and $\bar{F}$, that are projectively equivalent on the manifold $M$. Their Riemann curvatures are related by
\be\label{Rieproj}
\bar{R}{^i}_k=R{^i}_k+E\delta{^i}_k+\tau_k y^i,
\ee
where
\[
E=P^2-P_{|m}y^m, \quad \quad \tau_k=3(P_{|k}-PP_{.k})+E_{.k}.
\]
Here $P_{|k}$ denotes the covariant derivative of projective factor $P$ with respect to $\bar{F}$.\\
\end{lem}
Finsler spaces play a crucial role in both mathematics and physics by providing a broader framework than Riemannian geometry, allowing for a more flexible definition of distances and curvature. A particular type of Finsler metric known as the (generalized) \((\alpha, \beta)\)-metrics has garnered significant interest in recent studies \cite{Sh2} and \cite{generalalpha-beta}. Among these, spherically symmetric Finsler metrics in \(\mathbb{R}^n\) constitute an important class that warrants recognition \cite{Sphericaly}.\\
Let $F$ be a Finsler metric on $\mathbb{B}^n(\upsilon) := \{x \in \mathbb{R}^n : |x| < \upsilon\}$. The Finsler metric $F$ is considered spherically symmetric if it satisfies the following condition
\[
F(x,y)=F(Ax,Ay),
\]
for all $x \in \mathbb{B}^n(\upsilon)$, $y \in T_x\mathbb{B}^n(\upsilon)$ and $A \in O(n)$, where \( O(n) \) denotes the orthogonal group. Let \( |\cdot| \) and \( \langle , \rangle \) represent the standard Euclidean norm and inner product on \( \mathbb{R}^n \), respectively. Huang and Mo demonstrated in \cite{Sphericaly} that a Finsler metric \( F \) on \( \mathbb{B}^n(\upsilon) \) exhibits spherical symmetry if and only if there exists a function \( \varphi: [0, \upsilon) \times \mathbb{R} \rightarrow \mathbb{R} \) such that it can be expressed as
\[
F= u \varphi (r, s),
\]
with $(x,y) \in T\mathbb{B}^n(v)\setminus \{0\}$, $u=\mid y\mid$, $r=\mid x\mid$, $v=<x,y>$ and $s=\frac{v}{u}$.

\section{Weakly-Weyl Finsler metrics}
In this section, we introduce a new projective invariant quantity in Finsler geometry, known as weakly-Weyl Finsler metrics. This quantity is a generalization of the well-known Weyl curvature, which is characterized by Finsler metrics of scalar flag curvature. Weakly-Weyl Finsler metrics are defined as Finsler metrics for which the curvature $\tilde{W}$, referred to as the weakly-Weyl curvature, vanishes. The weakly-Weyl curvature is defined as follows.
\begin{Def}\label{DefW-Weyl}
For a Finsler space $(M,F)$ with the Weyl curvature $W=\{W_y\}_{y \in T_xM\setminus{0}}$, the weakly-Weyl curvature is defined as,
\be\label{W-Weyl}
\tilde{W}_y: T_xM \times T_xM \times T_xM \longrightarrow T_xM\\
\ee
\[
\tilde{W}_y (u, v, w)= \tilde{W}_j{^i}_{kl}u^j v^k w^l \frac{\partial} {\partial x^i},
\]
where $\tilde{W}_j{^i}_{kl}=W_j{^i}_{pl.k}y^p$ and $W_j{^i}_{pl}=\frac{1}{3}(W{^i}_{p.l}-W{^i}_{l.p})_{.j}$.
\end{Def}
Based on the definition provided above, we can conclude the following proposition
\begin{prop}
A Finsler metric has a vanishing weakly-Weyl curvature vanishes if and only if it is of $W$-quadratic.
\end{prop}
\begin{proof}
Assume that the Finsler metric \( F \) has a vanishing weakly-Weyl curvature. This implies that \( W_j{^i}_{pl.k} y^p = 0 \). Differentiating this expression with respect to \( y^m \), we obtain
\[
0= (W_j{^i}_{pl.k}y^p)_{.m}= W_j{^i}_{ml.k} + W_j{^i}_{pl.k.m}y^p= W_j{^i}_{ml.k} +( W_j{^i}_{pl.m}y^p)_{.k}- W_j{^i}_{kl.m}.
\]
Since we have \( W_j{^i}_{pl.m} y^p = 0 \), the equation simplifies to
\[
W_j{^i}_{ml.k} - W_j{^i}_{kl.m}=(W{^i}_{ml.k} - W{^i}_{kl.m})_{.j}= \frac{1}{3} (W{^i}_{m.k} - W{^i}_{k.m})_{.l.j} =0.
\]
This indicates the existence of a tensor \( \Omega_l{^i}_{km}= \Omega_l{^i}_{km}(x) \), with the property $\Omega_l{^i}_{km}=-\Omega_l{^i}_{mk}$, such that
\[
\frac{1}{3}(W{^i}_{m.k} - W{^i}_{k.m})_{.l}=\Omega_l{^i}_{km}(x).
\]
By contracting the above equation with \( y^l \) and \( y^m \), we derive
\[
W{^i}_{k} = -\Omega_l{^i}_{km}(x) y^l y^m,
\]
which completes the proof of the proposition.
\end{proof}
Then we define
\begin{Def}\label{DefW-WeylMetric}
A Finsler metric $F$ is called a weakly-Weyl Finsler metric if its weakly-Weyl curvature satisfies the following equation,
\[
\tilde{W}_j{^i}_{kl}=\omega_{jkl} y^i,
\]
where \( \omega_{jkl}=\omega_{jkl}(x,y) \) represents the tensor coefficients, and it holds that
\be\label{DefomegaW-Weyl}
\omega_{jkl|0} + \mu F\omega_{jkl}=0,
\ee
for a scalar function \( \mu \) defined on the tangent bundle \( TM \).
\end{Def}
Given that Weyl curvature is a projective invariant quantity in Finsler spaces, and recognizing the concept of weakly-Weyl curvature, it is simple to establish the subsequent Theorem.
\subsection{\textbf{Proof of Theorem \ref{W-WeylProjInv}}}
\begin{proof}
Consider the Finsler metrics $F$ and $\bar{F}$, which are projectively related through their geodesic coefficients, denoted by $G^{i}$ and $\bar{G}^{i}$, respectively. We have
\be\label{ProjG}
\bar{G}^{i}=G^{i}+P y^{i},
\ee
with projective factor $P$. After differentiating with regards to $y^{j}$, we will have
\begin{equation}\label{ProjGij}
\bar{G}^{i} \cdot{ }_{j}=G^{i}{ }_{. j}+P \delta^{i}{ }_{j}+P_{. j} y^{i},
\end{equation}
The Weyl curvature of both metrics, \( F \) and \( \bar{F} \), is identical. According to Definition \ref{DefW-WeylMetric}, it can be expressed as
\begin{equation}\label{Weylomega}
\bar{W}_j{^i}_{ml.k}y^m = W_j{^i}_{ml.k}y^m = \omega_{jkl} y^i,
\end{equation}
with the condition that \( \omega_{jkl|0} =- \mu F \omega_{jkl} \) for some scalar function \( \mu \) defined on the tangent bundle \( TM \). We will now demonstrate that
\[
\omega_{jkl||0} +  \bar{\mu} \bar{F} \omega_{jkl}=0,
\]
for a scalar function \( \bar{\mu} \) on \( TM \). Here, the symbol \( || \) denotes the horizontal derivative with respect to the metric \( \bar{F} \). Based on equations \eqref{ProjG} and \eqref{ProjGij}, we obtain
\[
\omega_{jkl||0} = \omega_{jkl;0} - 2 (G^{r}+P y^{r}) \omega_{jkl.r} - (G^{r}_{. j}+P \delta^{r}_{j}+P_{. j} y^{r}) \omega_{rkl} - (G^{r}_{. k}+P \delta^{r}_{k}+P_{. k} y^{r}) \omega_{jrl}
\]
\[
- (G^{r}_{. l}+P \delta^{r}_{l}+P_{. l} y^{r}) \omega_{jkr}.
\]
According to equation \eqref{Weylomega}, we find that \( \omega_{jkl.r} y^r = -\omega_{jkl} \) and \( \omega_{jkl} y^l = \omega_{jkl} y^k = \omega_{jkl} y^j = 0 \). Thus, from the above equation, we get
\[
\omega_{jkl||0} = \omega_{jkl|0}-P\omega_{jkl}.
\]
According to assumption, $F$ is a weakly-Weyl metric, then there is a scalar function $\mu$ on $TM$ such that $\omega_{jkl|0} + \mu F \omega_{jkl}=0$. Then we conclude
\[
\omega_{jkl||0} = - (\mu F + P) \omega_{jkl} = - \bar{\mu} \bar{F} \omega_{jkl}.
\]
This completes the proof.
%which by the \eqref{Weylomega}, one gets
\end{proof}
According to the Definition \ref{DefW-WeylMetric}, it is straightforward to deduce the following proposition.
\begin{prop}
Weyl Finsler metrics and $W$-quadratic Finsler metrics are weakly-Weyl metrics.
\end{prop}

The class of weakly-Weyl metrics includes significant types of projective invariant Finsler metrics, such as Weyl and $W$-quadratic Finsler metrics.
In \cite{W-WeylnotWeyl}, the authors introduce some spherically symmetric Finsler metrics that are not Weyl metrics. One can see that they are (non-trivial) weakly-Weyl metrics and also Douglas metrics.
\begin{ex}\label{W-WeylnotWeyl}\cite{W-WeylnotWeyl}
Spherically symmetric Finsler metrics in $\mathbb{B}^{n}(\upsilon) \subseteq \mathbb{R}^n$, has been introduced as $F=|y| \varphi(|x|,\frac{<x,y>}{|y|})$ with $\varphi:[0, \upsilon) \times \mathbb{R} \rightarrow \mathbb{R}$ where $(x,y) \in T \mathbb{B}^{n}(\upsilon)\setminus \{0\}$ and
\[
\varphi(r, s)= s. h(r)-\frac{s}{(a+br^2)^{\lambda}}\int_{s_0}^s \sigma^{-2}f\Big(\frac{r^2-\sigma^2}{(a+br^2)^{\lambda}}\Big)d\sigma,
\]
where $a$, $b$ and $\lambda$ are constants satisfying $a+br^2 >0$. These metrics are not Weyl metrics but are weakly-Weyl and also Douglas metrics. Their Weyl curvature is as follows
\[
W{^i}_{l}=\frac{4\lambda(\lambda-1)b^2}{(a+br^2 )^2}\big( x_j x^i\delta_{kl} +\frac{1}{n-1}x_k x_l \delta{^i}_j +\frac{|x|^2}{n-1}\delta{^i}_k\delta_{jl}-\delta_{jl}x_k x^i
\]
\[
- \frac{|x|^2}{n-1}\delta{^i}_j\delta_{kl}-\frac{1}{n-1} x_j x_k\delta{^i}_{l} \big)y^k y^l.
\]
\end{ex}
Then one could concludes
\begin{prop}
The class of Weyl Finsler metrics is a proper subset of the class of weakly-Weyl Finsler metrics.
\end{prop}
The theorem that will be proved in Sub-section \ref{Spherically} demonstrates that a symmetric Finsler metric in $\mathbb{R}^n$ is weakly-Weyl if and only if it is of $W$-quadratic. It is worthwhile to consider other particular classes of Finsler metrics, such as Randers metrics, more generally, $(\alpha, \beta)$-metrics or Generalized $(\alpha, \beta)$-metrics.\\
For the next step of our research, it is now interesting to explore the connections between this novel class of Finsler metrics and other projective invariants, such as Douglas and $GDW$ spaces. To initiate this exploration, we present the following theorem, which not only proves Sakaguchi's remarkable theorem—that all Weyl metrics (metrics with vanishing Weyl curvature) must be of GDW type \cite{Sakaguchi}—but also generalizes it.
\subsection{Proof of Theorem \ref{PreSakaguchi}}
\begin{proof}
The Weyl curvature of a Finsler metric $(M, F)$ is defined as \cite{Sh2}
\[
W{^i}_k= A{^i}_k -\frac{1}{n+1} A{^m}_{k.m} y^i,
\]
where $A{^i}_k=R{^i}_k-R \delta{^i}_{k}$ and $R=\frac{1}{n-1} R{^m}_m$. From this definition, we can express the Riemann curvature tensor as follows.
\[
R{^i}_k= W{^i}_k+ R \delta{^i}_k + \frac{1}{n+1} A{^m}_{k.m} y^i.
\]
Substituting this expression into \eqref{Rikl}, we obtain
\be\label{RjiklWeyl}
\begin{aligned}
3R_j{^i}_{ml}=(W{^i}_{m.l}-W{^i}_{l.m})_{.j}+ (\frac{1}{n+1}A{^s}_{m.s}-R_{.m})_{.j}\delta{^i}_l-(\frac{1}{n+1}A{^s}_{l.s}\\
-R_{.l})_{.j}\delta{^i}_m
+\frac{1}{n+1}(A{^s}_{m.l}-A^{s}_{l.m})_{.s}\delta{^i}_j+\frac{1}{n+1}(A{^s}_{m.l}-A{^s}_{l.m})_{.s.j}y^i.
\end{aligned}
\ee
However, per the definition of $A{^i}_k$, we can see that
\[
\frac{1}{n+1}A{^s}_{k.s}-R_{.k}=\frac{1}{n+1}(R{^s}_{k.s}-(n+2)R_{.k}),
\]
and
\[
A{^s}_{k.l}-A{^s}_{l.k}= 3 R{^s}_{kl} -(R_{.l} \delta{^s}_k-R_{.k} \delta{^s}_l).
\]
To compute $R_j{^i}_{ml.k}$, we first differentiate \eqref{RjiklWeyl} with respect to $y^k$. Then, by substituting the resulting equations into $R_j{^i}_{ml.k}$, we obtain the following expression.
\be\label{RjiklmWeyl}
\begin{aligned}
3R_j{^i}_{ml.k}=3W_j{^i}_{ml.k}+\frac{1}{n+1}\big[(R{^s}_{m.s}-(n+2)R_{.m})_{.j.k}\delta{^i}_l-(R{^s}_{l.s}\\ -(n+2)R_{.l})_{.j.k}\delta{^i}_m
+ 3R_s{^s}_{ml.k}\delta{^i}_j + 3R_s{^s}_{ml.j}\delta{^i}_k + 3R_s{^s}_{ml.j.k}y^i\big],
\end{aligned}
\ee
where $ W_j{^i}_{kl}=\frac{1}{3}(W{^i}_{k.l}-W{^i}_{l.k})_{.j} $. By utilizing the Ricci identity \eqref{RieBer}, we can express the relations as follows.
\be \label{RicciidRieE}
R_j{^i}_{ml.k} y^m=B_j{^i}_{kl|0}, \qquad R_s{^s}_{ml.k}= 2(E_{kl|m}-E_{km|l}).
\ee
Combining equations \eqref{Rikl} and \eqref{RieBer}, and taking into account the previous equation, we derive
\be \label{RicciidRieH}
\begin{aligned}
R_s{^s}_{ml.k}y^m= \frac{1}{3} R{^s}_{m.s.l.k}y^m = 2 H_{kl}.
\\
R_s{^s}_{ml.j.k}y^m=2H_{jl.k}-R_s{^s}_{kl.j}=2H_{jl.k}-2(E_{jl|k}-E_{jk|l})\\
=2E_{jl|p.k}y^p+ 2 E_{jl|k}-2(E_{jl|k}-E_{jk|l})=2(E_{jl.k|0}+E_{jk|l}).
\end{aligned}
\ee
Substituting equations \eqref{RicciidRieE} and \eqref{RicciidRieH} into the contracted form of \eqref{RjiklmWeyl} by $y^m$, we obtain
\be
\begin{aligned}\label{RjiklmWeylym}
W_j{^i}_{ml.k}y^m=R_j{^i}_{ml.k} y^m - \frac{1}{n+1}\big[2 H_{jk} \delta{^i}_l+ 2H_{kl}\delta{^i}_j + 2 H_{jk}\delta{^i}_k
\\
+2E_{jl.k|0}y^i+(2E_{jk|l}-\frac{1}{3}(R{^s}_{l.s}-(n+2)R_{.l})_{.j.k}) y^i\big]
\end{aligned}
\ee
Utilizing the above equation, \eqref{RicciidRieE} and \eqref{D2} in the equation \eqref{RjiklmWeylym}, one gets \eqref{WeylDotheta}.
\end{proof}
The direct outcome of the Lemma stated earlier is the following Theorem.
\subsection{\textbf{Proof of Theorem \ref{GSakaguchi}}}
\begin{proof}
Assume that $F$ is a weakly-Weyl metric. Then $W_j{^i}_{ml.k}y^m=\omega_{jkl}y^i$, which is consistent with the equation \eqref{DefomegaW-Weyl}. Using the equation \eqref{WeylDotheta} one obtains
\[
D_j{^i}_{kl|0} =(\omega_{jkl}+ \frac{1}{n+1} \theta_{jkl}) y^i.
\]
This implies that $F$ is a $GDW$-metric.
\end{proof}
When evaluating the equation represented by \eqref{WeylDotheta}, we can deduce that it is not always the case that every weakly-Weyl Finsler metric is a Douglas metric. The subsequent example serves as evidence for this assertion.
\begin{ex}\label{W-WeylnotDouglas} \cite{Osaka}
Put
\[
\Omega=\{(x,y,z) \in R^3 | x^2+y^2+z^2 <1\}, \quad p=(x,y,z) \in \Omega, \quad y=(u,v,w) \in T_p\Omega.
\]
Define the Randers metric $F=\alpha+\beta$ by
\[
\alpha=\frac{\sqrt{(-yu+xv)^2+(u^2+v^2+w^2)(1-x^2-y^2)}}{1-x^2-y^2}, \quad \beta=\frac{-yu+xv}{1-x^2-y^2}.
\]
The above Randers metric has vanishing flag curvature $K=0$ and S-curvature $S=0$. $F$ has zero Weyl curvature then $F$ is a weakly-Weyl metric. But $\beta$  is not closed then $F$ is not of Douglas type.
\end{ex}
Then, these two categories of Finsler metrics exhibit overlap in certain instances. The Example \ref{W-WeylnotWeyl}, shows that there are some Finsler metrics which are either Douglas and weakly-Weyl metrics. This metric belongs to the class of spherically symmetric Finsler metrics in $\mathbb{R}^n$. In the following, we consider these metrics in the case of weakly-Weyl Finsler metrics.
\subsection{Weakly-Weyl spherically symmetric Finsler metrics in $\mathbb{R}^n$}\label{Spherically}
Here, we delve into the intriguing class of weakly-Weyl spherically symmetric Finsler metrics in $\mathbb{R}^n$.  We begin by establishing a necessary and sufficient condition for a this special class of Finsler metrics to be weakly-Weyl. In particular, we prove a remarkable theorem that characterizes these metrics as being $W$-quadratic, a special class of Finsler metrics with strong projective invariance.\\
\subsection{\textbf{Proof of Theorem \ref{W-WeylSphThm}}}
\begin{proof}
Assume that $F$ is a spherically symmetric Finsler metric in $\mathbb{R}^n$ with the geodesic coefficients $G^i$ given by %\cite{SphGeodesic}
\[
G^i=uPy^i+ u^2Q x^i,
\]
where $P$ and $Q$ defined in \cite{SphGeodesic}. The Weyl curvature has beeb calculated in \cite{W-WeylnotWeyl}, as
\be\label{WeylofSpherically}
W{^i}_k=u^2 \omega_1 \delta{^i}_k+ (u^2 \omega_2 x_k+ u \omega_3 y_k) x^i+ (u \omega_4 x_k + \omega_5 y_k) y^i,
\ee
where
\[
\omega_1=-\frac{r^2-s^2}{n-1}R_2, \quad \omega_2=R_2,
\]
\[
\omega_3=-s R_2, \quad \omega_4= -\frac{s}{n-1}R_2-\frac{n-2}{n^2-1}(r^2-s^2)(R_2)_s,
\]
\[
\omega_5= \frac{r^2}{n-1}R_2+\frac{n-2}{n^2-1}s(r^2-s^2)(R_2)_s,
\]
and $R_2= 2Q(2Q-sQ_s)+\frac{1}{r}(2Q_{r}-sQ_{rs}-rQ_{ss})+(r^2-s^2)(2QQ_{ss}-Q_s^2)$. Note that
\be\label{WeylcoefSph}
\omega_5 + s \omega_4 = -\omega_1.
\ee
According to the relation between $\omega_2$ and $\omega_3$, we can rewrite \eqref{WeylofSpherically} as
\be\label{WeylofSpherically-Rewrite}
W{^i}_p=u^2 \omega_1 \delta{^i}_p+ u^3 \omega_2 s_{.p} x^i+ (u \omega_4 x_p + \omega_5 y_p) y^i.
\ee
Now, we aim to solve the equation $W_j{^i}_{pl.k}y^p=0$. First, we have
\[
3W{^i}_{pl}= W{^i}_{p.l}-W{^i}_{l.p}
\]
\[
=(uX_1 x_l + X_2 y_l)\delta{^i}_p- (uX_1 x_p + X_2 y_p)\delta{^i}_l+ 3\omega_2 (x_py_l-x_ly_p)x^i + \frac{1}{u}X_3(x_py_l-x_ly_p)y^i,
\]
where $X_1=(\omega_1)_s-\omega_4$, $X_2=2\omega_1-s(\omega_1)_s-\omega_5$ and $X_3=\omega_4-s(\omega_4)_s-(\omega_5)_s$. Then
\be\label{Wjipl}
3W_j{^i}_{pl}= A_{lj}\delta{^i}_p- A_{pj}\delta{^i}_l + E_{pl}\delta{^i}_j + 3B_{plj}x^i + D_{plj}y^i,
\ee
where
\[
A_{lj}=(uX_1 x_l + X_2 y_l)_{.j}=A_{l.j},
\]
\[
B_{plj}=\omega_2 (x_p\delta_{jl}-x_l\delta_{jp})+(\omega_2)_s (x_py_l-x_ly_p) s_{.j}=B_{pl.j},
\]
\[
D_{plj}=\frac{1}{u}X_3(x_p\delta_{jl}-x_l\delta_{jp})+(\frac{(X_3)_s}{u} s_{.j}-\frac{X_3}{u^3}y_j) (x_py_l-x_ly_p)= E_{pl.j},
\]
where
\[
E_{pl}=\frac{1}{u}X_3(x_py_l-x_ly_p).
\]
Then according to the Definition \ref{DefW-WeylMetric}, we have
\[
3W_j{^i}_{pl.k}y^p = - A_{pj.k} y^p \delta{^i}_l + E_{pl.k} y^p \delta{^i}_j + D_{plj}y^p \delta{^i}_k+ 3B_{plj.k} y^p x^i +[A_{lj.k} + D_{plj.k} y^p] y^i=\omega_{jkl} y^i,
\]
with $\omega_{jkl|0}+ \mu F \omega_{jkl}=0$, for some scalar function $\mu$ on $TM$. It means that
\begin{align}
%\right
A_{pj.k} y^p=0,\label{eq1}\\
D_{plj}y^p=0, \label{eq2}\\
E_{pl.k} y^p=0,\label{eq3}\\
B_{plj.k} y^p=0, \label{eq4}\\
A_{lj.k} + D_{plj.k} y^p=\omega_{jkl},\label{eq5}
\end{align}
According to \eqref{eq3}, we get
\[
E_{pl.k} y^p = s X_3 \delta_{lk} - (X_3)_s x_l x_k + \frac{s}{u} (X_3)_s (x_l y_k + x_k y_l)-\frac{s}{u^2} (X_3+s(X_3)_s) y_k y_l=0,
\]
which implies that
\[
X_3=\omega_4-s(\omega_4)_s-(\omega_5)_s=0.
\]
It means that
\be\label{E=0}
E_{kl}=0.
\ee
Applying the equation $X_3=0$ to the formula for $D_{plj}$, we conclude that
\be\label{D=0}
D_{plj} = 0.
\ee
Combining this with \eqref{eq5}, we find that
\be \label{diffA=0}
A_{lj.k} = \omega_{jkl}.
\ee
On the other hands, based on \eqref{eq4}, we get
\[
B_{plj.k}y^p= u(\omega_2)_s\big[(s\delta_{jl}-\frac{1}{u}x_l y_j)s_{.k}+(s\delta_{lk}-\frac{1}{u}x_l y_k)s_{.j}- u^2 s_{.j.k} s_{.l}\big]-u^3(\omega_2)_{ss}s_{.l} s_{.j}s_{.k}=0.
\]
Then all coefficients and the coefficient of the term $\delta_{jl}x_k$ vanish. This implies that $(\omega_2)_s=0$ and noting the formula of $B_{klj}$ presented in \eqref{Wjipl}, we have
\[
B_{klj}= \omega_2(r) (x_k\delta_{jl}-x_l\delta_{jk}).
\]
It means that
\be\label{diffB=0}
B_{klj.m}=0.
\ee
Noting that \( \omega_2 = \omega_2(r) \) and considering the equations presented after \eqref{WeylofSpherically}, we find that
\[
\omega_1=-\frac{r^2-s^2}{n-1}\omega_2(r), \quad \omega_3=-s\omega_2(r), \quad \omega_4=-\frac{s}{n-1}\omega_2(r),
\]
\[
\omega_5=\frac{r^2}{n-1}\omega_2(r).
\]
Thus
\[
X_1=\frac{3s}{n-1}\omega_2(r), \quad X_2=-\frac{3r^2}{n-1}\omega_2(r).
\]
Then using the above equations, \eqref{diffB=0}, \eqref{D=0} and \eqref{E=0} in \eqref{Wjipl}, we get
\be\label{Wjiklsphw-quad}
W_j{^i}_{kl}= \omega_2(r)\Big( \frac{1}{n-1}(x_j x_l-r^2\delta_{jl})\delta{^i}_k-\frac{1}{n-1}(x_j x_k-r^2\delta_{jk})\delta{^i}_l+(x_k \delta_{jl}-x_l\delta_{jk})x^i\Big).
\ee
It concludes the proof.
\end{proof}
In this section, using this novel class, weakly-Weyl metrics, we introduce another new class of projective invariant Finsler metrics which is not necessarily a subset of the class of $GDW$-metrics while having an intersection with it.
\section{Generalized weakly-Weyl Finsler metrics}
In the next section, we introduce a novel class of projective invariant Finsler metrics, generalized weakly-Weyl Finsler metrics, that is constructed using the new class of weakly-Weyl metrics. This new class is not necessarily a subset of the class of $GDW$-metrics, but it does exhibit an intersection with that class. The defining characteristics of this new class set it apart from the existing categories of Finsler metrics, offering a unique perspective and potential applications. The approach used to construct this new class of Finsler metrics is similar to the approach employed in the paper \cite{SadeDouglas} for Douglas curvature. This new class of generalized weakly-Weyl Finsler metrics, along with the class of relatively isotropic $\tilde{D}$-metrics introduced in \cite{SadeDouglas}, will pave the way for the search of a novel projective invariant Finsler class that encompasses the class of $GDW$-metrics. Now, we present the definition for generalized weakly-Weyl Finsler metrics.
\begin{Def}
A Finsler metric $F$ is called generalized weakly-Weyl Finsler metric if its weakly-Weyl curvature satisfying the following equation.
\be\label{DefRIWW}
\tilde{W}_j{^i}_{kl|0}+ \mu F \tilde{W}_j{^i}_{kl} = \lambda_r \tilde{W}_j{^r}_{kl} y^i,
\ee
for some tensors $\lambda_r$ and smooth scalar function $\mu$ on $TM$. The weakly-Weyl curvature $\tilde{W}_j{^i}_{kl}$ have been introduced in Definition \ref{DefW-Weyl}.
\end{Def}
First of all, we show that the class of generalized weakly-Weyl ($GWW(M)$) Finsler metrics is closed under projective changes.  More precisely, if $F$ is projectively equivalent to $\bar{F} \in GWW(M)$ then $F\in GWW(M)$.
\subsection{\textbf{Proof of Theorem \ref{GW-WeylProjInv}}}
\begin{proof}
Assume that $F$ is projective equivalent to the other Finsler metric $\bar{F}$ with the same Weyl curvature $W=\{W_y\}_{y \in T_xM\setminus{0}}$.  According to \eqref{ProjGeo}, we have
\be\label{barGandG}
\bar{G}^i=G^i+ P y^i,
\ee
and then
\be\label{barNandN}
\bar{N}{^i}_j=N{^i}_j+ P\delta{^i}_j+ P_{.j} y^i.
\ee
If $\tilde{W}_j{^i}_{kl||0}$ denotes the horizontal covariant derivatives with respect to the Berwald connection of $\bar{F}$, then we have
\[
\tilde{W}_j{^i}_{kl||0}= \tilde{W}_j{^i}_{kl;m} y^m- 2 \tilde{G}^r \tilde{W}_j{^i}_{kl.r}- \tilde{W}_r{^i}_{kl}\bar{N}{^r}_j - \tilde{W}_j{^i}_{rl}\bar{N}{^r}_k
- \tilde{W}_j{^i}_{kr}\bar{N}{^r}_l + \tilde{W}_j{^r}_{kl}\bar{N}{^i}_r.
\]
Using Equations \eqref{barGandG} and \eqref{barNandN} in the previous expression, and noting that the weakly-Weyl curvature tensor $\tilde{W}_j{^i}_{kl}$ is homogeneous of degree zero in $y$ and $\tilde{W}{^i}_{kl}$ is homogeneous of degree 1 in $y$, we obtain
\[
\tilde{W}_j{^i}_{kl||0}= \tilde{W}_j{^i}_{kl|0}+ P_{.r} \tilde{W}_j{^r}_{kl} y^i -2P \tilde{W}_j{^i}_{kl}.
\]
Since $F$ is a generalized weakly-Weyl Finsler metric, there exist tensors $\lambda_r$, such that
\[
\tilde{W}_j{^i}_{kl|0}=-\mu F \tilde{W}_j{^i}_{kl} + \lambda_r  \tilde{W}_j{^r}_{kl} y^i.
\]
Substituting this into the previous equation, we obtain
\[
\tilde{W}_j{^i}_{kl||0}= -(\mu F  + 2 P) \tilde{W}_j{^i}_{kl} + (\lambda_r+ P_{.r}) \tilde{W}_j{^r}_{kl} y^i= -\bar{\mu}\bar{F}\tilde{W}_j{^i}_{kl}+\bar{\lambda}_r \tilde{W}_j{^r}_{kl} y^i,
\]
which means that $\bar{F}$ is a generalized weakly-Weyl Finsler metric.
\end{proof}
The above Proposition demonstrates that this new class of Finsler metrics is projective invariant. To gain a deeper understanding of this new class, it is essential to explore its connections with other projective invariant quantities in Finsler spaces.
\begin{prop}
For every generalized weakly-Weyl Finsler metric, we have
\be\label{GWWandDouglas}
D_j{^i}_{kl|0|0} + \mu F D_j{^i}_{kl\mid 0} = T_{jkl}y^i,
\ee
for some tensors $T_{jkl}$ and smooth scalar function $\mu$ on $TM$.
\end{prop}
\begin{proof}
Assume that $F$ is a generalized weakly-Weyl Finsler metric. Then for some tensors $\lambda_r$ and smooth scalar function $\mu$, we have
\[
\tilde{W}_j{^i}_{kl|0}+ \mu F \tilde{W}_j{^i}_{kl} = \lambda_r  \tilde{W}_j{^r}_{kl} y^i.
\]
Using Theorem \ref{WeylDouglasTheta} with noting the above equation, we obtain
\[
D_j{^i}_{kl|0|0}+ \mu F D_{j}{^i}_{kl\mid 0}-  \lambda_r D_j{^r}_{kl|0} y^i=\frac{1}{n+1} \big(\theta_{jkl|0} +(\mu F- \lambda_0) \theta_{jkl}\big) y^i.
\]
By putting $T_{jkl}=\lambda_r D_j{^r}_{kl|0}+\frac{1}{n+1} \big(\theta_{jkl|0} +(\mu F- \lambda_0) \theta_{jkl}\big)$, one finds \eqref{GWWandDouglas}. $\theta_{jkl}$ has been defined in the Theorem \ref{WeylDouglasTheta}.
\end{proof}

\subsection{Concluding Remark}

By the equation \eqref{WeylDotheta} in Theorem \ref{WeylDouglasTheta}, we have
\[
\tilde{W}_j{^i}_{kl\mid 0}+ \mu F \tilde{W}_j{^i}_{kl} -  \lambda_r \tilde{W}_j{^r}_{kl} y^i
\]
\[
=D_j{^i}_{kl|0|0}+ \mu F D_{j}{^i}_{kl\mid 0}-  \lambda_r D_j{^r}_{kl|0} y^i-\frac{1}{n+1} \big(\theta_{jkl|0} +(\mu F- \lambda_0) \theta_{jkl}\big) y^i
\]
The above equation suggests that there may exist some generalized weakly-Weyl Finsler metrics that are not $GDW$-metrics. It also shows that there may exist a $GDW$-metric $F$ with $D_j{^i}_{kl|0}=d_{jkl}y^i$, that is not a generalized weakly-Weyl Finsler metric. Based on the equation \eqref{WeylDotheta}, for this metric we have
\[
\tilde{W}_j{^i}_{kl|0}+ \mu F \tilde{W}_j{^i}_{kl}- \lambda_r \tilde{W}_j{^r}_{kl}y^i= \big(\tilde{d}_{jkl\mid 0}+(\mu F- \lambda_0)\tilde{d}_{jkl}\big)y^i,
\]
where $\tilde{d}_{jkl}=d_{jkl}-\frac{1}{n+1}\theta_{jkl}$ and $\lambda_0=\lambda_r y^r$. The formula for $\theta_{jkl}$ is presented in the Theorem \ref{WeylDouglasTheta}.\\
The class of generalized weakly-Weyl Finsler metrics is a projective invariant class that has non-empty intersections with the class of $GDW$-metrics, such as the class of weakly-Weyl metrics. However, generalized weakly-Weyl metrics are not necessarily a subset of or contained within the class of $GDW$-metrics.
\section*{Declarations}
\subsection*{\textbf{Competing Interests}}
\quad The authors declare that they have no competing interests related to the work submitted for publication.
\subsection*{\textbf{Funding}}
\quad The authors acknowledge that there were no financial supports received for this study.
\subsection*{\textbf{Ethical Approval}}
\quad This research did not involve human participants or animals, and therefore ethical approval was not required.
\subsection*{\textbf{Informed Consent}}
\quad As this study did not involve human participants, informed consent was not applicable.

\
%\bibliography{sn-bibliography}

\end{document}